\documentclass[10pt,reqno]{amsart}
\usepackage[svgnames,table]{xcolor} 
\usepackage[colorlinks=true,citecolor=Magenta,linkcolor=cyan, urlcolor=Green, pdftoolbar=true]{hyperref} 
\usepackage{amsmath,amsfonts,latexsym,amssymb}
\usepackage{amssymb,amsfonts, amsmath}
\usepackage{enumitem,overpic,caption,amssymb}

\usepackage{mathrsfs}
\usepackage[latin1]{inputenc}
\usepackage[T1]{fontenc}
\usepackage{ae,aecompl}
\usepackage{braket}
\usepackage{comment}
\usepackage{mathtools} 
\usepackage[all,arc]{xy}
\usepackage{mathrsfs}
\usepackage{color}
\usepackage{subfig}
\usepackage{verbatim} 
\usepackage[all,arc]{xy}
\usepackage{mathrsfs}
\usepackage{hyperref}
\usepackage{xstring}
\usepackage[makeroom]{cancel}
\usepackage{mathabx}
\usepackage{braket}
\makeatletter
\newsavebox{\@brx}
\newcommand{\llangle}[1][]{\savebox{\@brx}{\(\m@th{#1\langle}\)}%
  \mathopen{\copy\@brx\mkern2mu\kern-0.8\wd\@brx\usebox{\@brx}}}
\newcommand{\rrangle}[1][]{\savebox{\@brx}{\(\m@th{#1\rangle}\)}%
  \mathclose{\copy\@brx\mkern2mu\kern-0.8\wd\@brx\usebox{\@brx}}}
\makeatother

\newtheorem{theorem}{Theorem}[section]
\newtheorem{lemma}[theorem]{Lemma}
\newtheorem{claim}[theorem]{Claim}
\newtheorem{proposition}[theorem]{Proposition}
\newtheorem{corollary}[theorem]{Corollary}

\renewcommand{\leq}{\leqslant}
\renewcommand{\geq}{\geqslant}

\newtheorem{rmk}[theorem]{\normalfont{\em{Remark}}}
\newtheorem{rmks}[theorem]{\normalfont{\em{Remarks}}}

\makeatletter
\renewcommand*\env@matrix[1][\arraystretch]{%
\edef\arraystretch{#1}%
\hskip -\arraycolsep
\let\@ifnextchar\new@ifnextchar
\array{*\c@MaxMatrixCols c}}
\makeatother

\title{Undecidability of Epimorphisms  onto products of hyperbolic groups}

\AtEndDocument{\bigskip{\footnotesize%
\textsc{Max Planck Institute for Mathematics in the Sciences, Inselstra{\ss}e 22, 04103 Leipzig, Germany} \par  
 \textit{E-mail address}: \texttt{konstantinos.tsouvalas@mis.mpg.de}}}
\author{Konstantinos Tsouvalas}
\begin{document}
\maketitle

\begin{abstract} We exhibit examples of finitely presented subgroups $P$ of direct products of hyperbolic groups for which there is no algorithm that detects whether a finitely presented group has a quotient isomorphic to $P$. For any torsion-free, linear, hyperbolic group $Q$ that maps onto the free group of rank $2$ and $m\geq 2$, we construct a recursive sequence $(\Gamma_n)_{n\in \mathbb{N}}$ of torsion-free, hyperbolic $C'(\frac{1}{6})$ small cancellation groups, with the property that there is no algorithm determining the values $n\in \mathbb{N}$ such that $\Gamma_n$ has a quotient isomorphic to the direct product $Q^{m}$ of $m$-copies of $Q$.  \end{abstract}

\section{Introduction}  \label{introduction}
The isomorphism problem is undecidable for the class of all finitely presented groups, as established by the Adian--Rabin theorem \cite{A,R}. In contrast, Sela \cite{Sela-iso} (see also \cite{Dahmani-Groves}) proved that the isomorphism problem is decidable among the class of all torsion-free  hyperbolic groups.  For the class of all hyperbolic groups, including those  with torsion elements, the isomorphism problem is decidable by \cite{Dahmani-Guirardel}. However, several other group theoretic properties have been proved to be undecidable among the class of hyperbolic groups; see for instance, \cite{BMS}, \cite[Cor. 3.8]{BO} and \cite[Thm. D]{Bridson-Wilton-triviality}. Bridson--Wilton \cite[Thm. A]{Bridson-Wilton-triviality} proved that there is no algorithm that, given a finite presentation, determines whether the group (defined by the presentation) admits a non-trivial finite quotient. The question whether such an algorithm exists among the class of torsion-free hyperbolic groups remains open and is, in fact,  equivalent to the existence of a \hbox{non-residually finite hyperbolic group \cite[Thm. 9.6]{Bridson-Wilton-triviality}.} For certain collections $\mathcal{P}$ of finite groups, the problem of deciding whether a finitely presented group has a quotient in $\mathcal{P}$ has been studied in \cite{BELS, Plesken} with both positive and negative results. 

In this article, we exhibit examples of infinite, finitely presented groups $P$ for which there is no algorithm that detects whether a finitely presented group has a quotient isomorphic to $P$.  For a group $G$ and $m\geq 2$, denote by $G^m$ the direct product of $m$-copies of $G$. For two groups $G_1$ and $G_2$, denote by $\mathsf{Epi}(G_1,G_2)$ the set of epimorphisms from $G_1$ to $G_2$; if $\mathsf{Epi}(G_1,G_2)$ is non-empty we say that $G_1$ maps onto $G_2$. A group $G$ is called {\em residually solvable} if for every non-trivial element $g\in G$ there is a solvable group $S$ and a homomorphism $\phi:H\rightarrow S$ with $\phi(g)\neq 1$.  Our first result is the following.

\begin{theorem} \label{mainthm2} Let $Q$ be a finitely presented, virtually residually solvable group that maps onto the free group of rank $2$, and the centralizer of any non-cyclic free subgroup of $Q$ is trivial. Fix $m \geq 2$ an integer. There is a finite set $Z$ and a recursive sequence $(\mathcal{R}_n)_{n\in \mathbb{N}}$ of finite subsets of words in $Z^{\pm 1}$, of cardinality independent of $n$, with the following properties:
\begin{enumerate}\item for every $n\in \mathbb{N}$ the group $$\Gamma_n=\big \langle Z\ |\ \mathcal{R}_n \big \rangle$$ is torsion-free, hyperbolic and satisfies the $C'(\frac{1}{6})$ small cancellation condition.

\item the set $$\big\{n\in \mathbb{N}: \mathsf{Epi}(\Gamma_n,	Q^{m})\neq \emptyset \big\}$$ is recursively enumerable but not recursive.\end{enumerate} \end{theorem}

 Since finitely generated linear groups are virtually residually nilpotent \cite{Platonov}, Theorem \ref{mainthm2} shows that there is no algorithm that detects whether a torsion-free hyperbolic group maps onto  $Q^m$, $m\geq 2$, where $Q$ is a linear, torsion-free hyperbolic group that maps onto the free group of rank $2$. In fact, we obtain the following slightly more general statement.

\begin{theorem} \label{maincor}  Let $Q$ be a torsion-free, hyperbolic group which maps onto the free group of rank $2$. For any $m\geq 2$, there is no algorithm that decides whether a finitely presented group has a quotient isomorphic to $Q^{m}$. \end{theorem}

Recall that a group $G$ is called of type $\mathcal{F}_n$ if it admits a classifying space with finite $n$-skeleton. The second main result of the present article provides further examples of finitely presented groups $P$ for which there is no algorithm that determines whether a finitely presented group has a quotient isomorphic to $P$.

\begin{theorem}\label{fiberprod} Let $G=\langle X \ | \ \mathcal{S}\rangle$ be an one-ended group of type $\mathcal{F}_3$ with unsolvable word problem and without finite quotients. Let $1\rightarrow N \rightarrow \Gamma {\rightarrow} G\ast G \rightarrow 1$ be a short exact sequence, where $N$ is finitely generated and $\Gamma$ is a torsion-free, linear, hyperbolic group. Let $\pi:\Gamma \rightarrow G\ast G$ be the projection with kernel $N$ and $P$ be the fiber product, $$P:=\big\{(\gamma_1,\gamma_2)\in \Gamma\times \Gamma: \pi(\gamma_1)=\pi(\gamma_2)\big\}.$$  There exists a finite set $Y$ and a recursive sequence of finite subsets $\big(\mathcal{S}_n)_{n\in \mathbb{N}}$, of cardinality independent of $n$, such that $\Delta_n:=\langle Y \ | \ \mathcal{S}_n\rangle$ is a torsion-free $C'(\frac{1}{6})$ small cancellation group and the set $$\big\{n\in \mathbb{N}:\mathsf{Epi}(\Delta_n,P)\neq \emptyset\big\}$$ is recursively enumerable but not recursive. 
\end{theorem}

As an immediate consequence of the previous theorem we obtain the following corollary.

\begin{corollary}\label{fiber-cor} Let $1\rightarrow N \rightarrow \Gamma \overset{\pi}{\rightarrow} G\ast G \rightarrow 1$ be a short exact sequence, where $\Gamma, N , G$ are as in Theorem \ref{fiberprod} and $P\subset \Gamma\times \Gamma$ is the fiber product of $\Gamma$ with respect to $N$. There is no algorithm that decides whether a torsion-free hyperbolic group has a quotient isomorphic to $P$.\end{corollary}

Let us note here that the fiber product $P\subset \Gamma \times \Gamma$ is finitely presented by the 1-2-3 theorem \cite{BBMS}. A plethora of examples of groups $G$ for which Theorem \ref{fiberprod} and Corollary \ref{fiber-cor} apply exist by \cite[Thm. 3.1]{Bridson-J.Alg.}. Moreover, the Rips algorithm \cite{Rips}, with imput group $G\ast G$, provides examples of pairs $(\Gamma,N)$ satisfying the assumptions of Theorem \ref{fiberprod}.

Let us remark that for a finitely presented group $W$, the property of a finitely presented group admitting an epimorphism onto $W$ is not a Markov property (e.g. see \cite[Def. 3.1]{Miller}), since, any finitely presented group $G$ can be embedded into the  free product $G\ast W$ and the latter group clearly surjects onto $W$. In particular, Theorems \ref{mainthm2}, \ref{fiberprod} and their corollaries are not direct consequences of the Adian--Rabin theorem \cite{A, R}. Nevertheless, one of the main tools that we use for the proof of Theorem \ref{mainthm2} is the undecidability result of Miller (see Theorem \ref{undecidability}) which is based on the existence of finitely presented groups with unsolvable word problem \cite{Novikov, Boone}. 

\subsection*{Acknowledgements} I would like to thank Sami Douba for helpful comments in an earlier version of the paper.

\section{ Rips's construction} \label{Rips2}
For  a non-empty set $X$ denote by $\mathsf{F}(X)$ the free group on $X^{\pm 1}$. For a subset $S\subset \mathsf{F}(X)$ denote by $\llangle S\rrangle=\big \langle \{gsg^{-1}:g\in \mathsf{F}(X), s\in S\}\big\rangle$ the intersection of all normal subgroups of $\mathsf{F}(X)$ containing $S$. For the definition of the $C'(\delta)$, $\delta \in \mathbb{Q}\cap(0,\infty)$, small cancellation condition we refer the reader to \cite[Ch. V]{LS}. 

The Rips construction \cite{Rips} provides an algorithm that takes as input a finite presentation of a group $W$ and outputs an explicit presentation of a torsion-free hyperbolic group $\Gamma$ and a short exact sequence $1\rightarrow N\rightarrow \Gamma\rightarrow W\rightarrow 1$, where $N$ is a normal $2$-generated subgroup of $\Gamma$. There are certain variations of this construction providing examples of hyperbolic groups with interesting properties, see, for instance, \hbox{\cite{Arenas, AS, BO,Rips-Wise}.} \par We provide below a refined version of the Rips construction that we use for the proof of Theorems \ref{mainthm2} and \ref{fiberprod}.

\begin{proposition}\label{Rips0} Let $\Lambda$ be a finitely generated group which contains a residually solvable subgroup of index $p \in \mathbb{N}$. There exists an algorithm that takes as input a finite presentation $W=\langle X\ |\ \mathcal{R}\rangle$ and constructs a torsion-free, hyperbolic, \hbox{$C'(\frac{1}{6})$ small cancellation group} $$\Gamma_W=\big \langle X\sqcup\{x,y\}\ |\ \mathcal{R}_W\big \rangle$$ where $|\mathcal{R}_W|=|\mathcal{R}|+4|X|+5$, satisfying the following properties: \begin{enumerate}
\item \label{item1-Rips} $N_W=\langle x,y\rangle$ is a normal subgroup of $\Gamma_W$ and there is a short exact sequence $$1\rightarrow N_W \rightarrow \Gamma_W \overset{\pi}{\rightarrow} W\rightarrow 1.$$

\item \label{item2-Rips} Any homomorphism $\varphi:\Gamma_W\rightarrow \Lambda$ factors as $\varphi=\varphi'\circ \pi$ for some homomorphism $\varphi':W\rightarrow \Lambda$ and the projection \hbox{$\pi:\Gamma_W\rightarrow W$.} 
\end{enumerate}\end{proposition}

\begin{proof} The proof follows the construction in  \cite{Rips} with slight modifications in order to ensure (\ref{item2-Rips}). For this, suppose that $W$ is given by the presentation \begin{equation}\label{presentation-W}W=\big \langle x_1,\ldots,x_r \ |\ R_1, \ldots,R_{s} \big \rangle\end{equation} where $R_1,\ldots,R_{s}$ are (possibly empty) words of the free group on $\{x_1,\ldots,x_r\}^{\pm 1}$. Let us set $m:=3^8 (p!)$ and consider the group $\Gamma_W$ given by the presentation \begin{equation*}\label{presentation1} \Gamma_W=\scalebox{3.2}{\Bigg \langle} \! \! \! \! \! \! \!  \! \! x_1,\ldots, x_r, x,y   \ \scalebox{3}{\Bigg{|}} \ \begin{matrix}[1.25]
xyxy^{2}\cdots xy^{81} \ \ \ \ \ \ \ \ \ \ \ \ \ \  \\
xy^{82}xy^{83}\cdots xy^{162} \ \ \ \ \ \ \ \  \  \\
xy^{m}xy^{m+1}\cdots xy^{2m+1} \ \ \  \\
xy^{3m}xy^{3m+1}\cdots xy^{4m+2}\\
xy^{5m}xy^{5m+1}\cdots xy^{6m+3}\\ 
\vspace{0.001cm} \end{matrix} \ \ \begin{matrix}[1.25] 
R_jxy^{\lambda j +1}xy^{\lambda j +2}\cdots xy^{\lambda j+\lambda} \ \ \ \ \ \  \ \ \ \ \ \\
x_{i}x x_{i}^{-1}xy^{k_{1i}+1}xy^{k_{1i}+2}\cdots xy^{k_{1i}+7m}\\ 
x_{i}yx_{i}^{-1}xy^{k_{2i}+1}xy^{k_{2i}+2}\cdots xy^{k_{2i}+7m}  \\ 
x_{i}^{-1}x x_{i}xy^{k_{3i}+1}xy^{k_{3i}+2}\cdots xy^{k_{3i}+7m}\\
x_{i}^{-1}yx_{i}xy^{k_{4i}+1}xy^{k_{4i}+2}\cdots xy^{k_{4i}+7m}\\  j=1\ldots,s, \ i=1\ldots,r \   \ \ \ \ \ \ \end{matrix} \! \! \! \!\scalebox{3.2}{\Bigg \rangle} \end{equation*} where $k_{di}:=7m(4i+d-4)$, $d=1,\ldots,4$ and $\lambda:=100 m r \big(|R_1|+\cdots+|R_s|\big)$, where $|R_j|$ is the word length of $R_j\in \mathsf{F}(\{x_1,\ldots,x_r\})$.

By the previous choices, $\Gamma_W$ satisfies the $C'(\frac{1}{9})$ small cancellation condition and is hyperbolic \cite{Greendlinger}. By construction, $N_W=\langle x,y\rangle$ is normal in $\Gamma_W$ and $ \Gamma_W /N_W=W$. Moreover, all relations used in the presentation of $\Gamma_W$ are no proper powers of words in $\mathsf{F}(\{a_1,\ldots, a_r,x,y\})$ and thus $\Gamma_W$ is torsion-free \cite{Greendlinger} \hbox{(see also \cite[Ch. V, Thm. 10.1]{LS}).} \par Now we check that every homomorphism from $N_W$ to $\mathsf{G}$ is trivial. \hbox{To see this, first note that} $$\mathsf{F}(\{x,y\})=\llangle y^{m}, xy^{m}xy^{m+1}\cdots xy^{2m+1}, xy^{3m}xy^{3m+1}\cdots xy^{4m+2}, xy^{5m}xy^{5m+1}\cdots xy^{6m+3} \rrangle,$$ hence $N_W=\langle \{hy^{m}h^{-1}:h\in N_W\}\rangle=\llangle y^m \rrangle$. In addition, $N_W=[N_W,N_W]$. To see this, observe from the first two defining relations of $\Gamma_W$ that $x^{3^4}y^{k},$ $x^{3^4}y^{3^8+k}\in [N_W,N_W]$, where $k=46\cdot 81$. This implies $y^{3^8}\in [N_W,N_W]$, $y^{\alpha}\in [N_W,N_W]$ and $[N_W,N_W]=\llangle y^{m}\rrangle$.

Now let $\varphi:N_W \rightarrow \Lambda$ be a homomorphism. Let also $\Lambda'$ be a residually solvable subgroup of $\Lambda$ with $p=[\Lambda:\Lambda']$. By the previous observation, $N_W$ cannot have a non-trivial finite-index subgroup of index at most equal to $p$,\footnote{If $N_W'$ is a subgroup with $[N_W:N_W']\leq p$, there is a normal finite-index subgroup $N_W''\triangleleft  N_W$, of index at most $p!$, contained in $N_W'$. In particular, $hy^{p!}h^{-1}\in N_W''$ for every $h\in N_W$, hence $N_W'=N_W''=N_W$.} hence, $\varphi(N_W)$ is a subgroup of $\Lambda'$. As $N_W=[N_W,N_W]$, any solvable quotient of $N_W$ is trivial and $\varphi$ is trivial. This shows that any homomorphism of $\Gamma_W$ to $\Lambda$ factors through the natural projection $\pi:\Gamma_W\rightarrow W$.\end{proof}

\begin{corollary}\label{domain} Let $A$ be a finitely generated domain and $W=\langle X \ | \ \mathcal{R}\rangle$ a finitely presented group. For every $\ell\geq 2$, there exists a torsion-free $C'(\frac{1}{6})$ small cancellation group $\Gamma_W=\langle Y \ | \ \mathcal{Q}\rangle$ with $|Y|=|X|+2$, $|\mathcal{Q}|=|\mathcal{R}|+4|X|+5$ and an epimorphism $\pi:\Gamma_W\rightarrow W$, such that any representation $\rho:\Gamma_W\rightarrow \mathsf{GL}_{\ell}(A)$ factors as $\rho=\rho'\circ \pi$ for some representation $\rho':W\rightarrow \mathsf{GL}_{\ell}(A)$.\end{corollary}

\begin{proof} By Platonov's theorem \cite{Platonov}, there is $p\in \mathbb{N}$ and a prime $q$, depending only on $A$, such that $\mathsf{GL}_{\ell}(A)$ contains a subgroup of index $p$ which is residually $q$-finite. Given $W$ and $p$, the output hyperbolic group $\Gamma_W$ of the algorithm of Proposition \ref{Rips0} has the property that for any homomorphism $\rho:\Gamma_W\rightarrow \mathsf{GL}_\ell(A)$, $N_W$ is contained in $\textup{ker}\rho$.\end{proof}

\begin{rmks} \normalfont{ \noindent \textup{(i)} By combined work of Agol \cite{Agol} and Wise \cite{Wise}, finitely presented groups satisfying the $C'(\frac{1}{6})$ small cancellation condition are linear over $\mathbb{Z}$. In fact, such groups also admit Anosov representations by \cite{DFWZ}. Corollary \ref{domain} provides explicit examples of torsion-free, linear hyperbolic groups all of whose representations, over a finitely generated domain $A$, factor through some representation of a given finitely presented group $W$. Examples of hyperbolic groups all of whose $\ell$-dimensional complex representations factor through a representation of a a given finitely generated group $G$, were exhibited by Osin--Belegradek by using their variation of the Rips construction \cite[Thm. 1.1 \& 3.1]{BO}. Let us note that these groups  are non-linear since they contain non-linear  subgroups that are quotients of quaternionic hyperbolic lattices \cite{Kap-nonlinear}.\\ \noindent \textup{(ii)} Given a finitely presented group $W$ and a field ${\bf K}$, there exists a Rips short exact sequence $1\rightarrow M_W\rightarrow \Delta_W\rightarrow W\rightarrow 1$, where $\Delta_W$ is a torsion-free, $C'(\frac{1}{6})$ small cancellation group such that any $\ell$-dimensional representation of $\Delta_W$ over ${\bf K}$ factors through a representation of $W$.  
\par To see this, for every $n\in \mathbb{N}_{\geq 5}\cup\{\infty\}$ consider the finitely generated group \begin{equation*} \Theta_n:=\scalebox{1.9}{\Bigg \langle}\!\!\! x,y   \ \scalebox{1.8}{\Bigg{|}} \ \begin{matrix}[1.15]
xy^{k !+1}xy^{k !+2}\cdots xy^{2k !} \ \ \ \ \ \ \ \\
xy^{2k !+1}xy^{2k !+2}\cdots xy^{3k !-1}  \\
xy^{3k !+1}xy^{3k!+2}\cdots xy^{4k !+1}\\
k=5,\ldots,n \ \ \ \  \end{matrix} \ \! \! \! \! \!\scalebox{1.9}{\Bigg \rangle}.\end{equation*} Note that $ \Theta_n$ satisfies the $C'(\frac{1}{9})$ small cancellation condition for $n\geq 5$. Since $ \Theta_{\infty}=\llangle y^{n!}\rrangle$ \hbox{for every $n \geq 5$,} $ \Theta_{\infty}$ has no proper subgroups of finite-index and any finite dimensional representation of $ \Theta_{\infty}$ (over any field) is trivial. Let $\overline{{\bf K}}$ be the algebraic closure of ${\bf K}$. Given $\ell \in \mathbb{N}$, by Hilbert's Nullstelensatz (see the discussion in the first paragraph of the proof of \cite[Thm. 3.1]{BO}) there exists $d\geq 5$, depending on $\ell$, such that every representation $\rho_d: \Theta_{d} \rightarrow \mathsf{GL}_{\ell}\big(\overline{{\bf K}}\big)$ factors through the natural projection $\pi_{d}: \Theta_{d} \rightarrow{}  \Theta_{\infty}$ and hence $\rho_d$ is trivial. By using the set of $3(d-4)$ relations of $ \Theta_{d}$ and proceeding as in Rips's algorithm, there is a short exact sequence $1\rightarrow M_W\rightarrow \Delta_W\rightarrow W\rightarrow 1$ such that $M_W$ is a quotient of $\Theta_d$ satisfying the prescribed properties.}\end{rmks}

\section{Proof of the Theorems}

 For a finitely generated group $\mathsf{\Gamma}$ its first Betti number is $b_1(\mathsf{\Gamma}):=\textup{dim}_{\mathbb{Q}}H_1(\Gamma,\mathbb{Q})$. A group $G$  is called Hopfian if every epimorphism from $G$ to itself is an automorphism of $G$. The following lemma, which is crucial for the proof of Theorems \ref{mainthm2} and \ref{maincor}, provides a necessary and sufficient criterion for certain finitely presented groups to admit an epimorphism onto direct products of finitely many copies of a \hbox{Hopfian group $Q$ that maps onto the free group of rank $2$.}

\begin{lemma}\label{mainlemma} Let $Q=\langle x_1,\ldots, x_r \ | \ \mathcal{R}\rangle$ be a Hopfian group which maps onto the free group of rank $2$ and the centralizer of any non-cyclic free subgroup of $Q$ is trivial. Let $\Sigma \subset \mathsf{F}(\{x_1,\ldots,x_r\})$ be a non-empty finite subset such that the group $\langle x_1,\ldots, x_r \ | \ \mathcal{R}\cup \Sigma \rangle$ has zero first Betti number. Fix $k\in \mathbb{N}$ and consider the finitely presented group \begin{align*} \label{presentation3} W_{\Sigma,k}:=\Big \langle x_1,\ldots,x_r, \{y_{1\sigma}\}_{\sigma \in \Sigma}, \ldots, \{y_{k\sigma}\}_{\sigma \in \Sigma}  \ \big| \ \mathcal{R}, [\sigma^{-1}y_{s\sigma}, y_{s\sigma'}], [y_{s \sigma},y_{t \sigma'}], \sigma,\sigma'\in \Sigma, s\neq t \Big \rangle.\end{align*} The group $W_{\Sigma,k}$ maps onto $Q^{k+1}$ if and only if $\big \langle x_1,\ldots,x_r \ | \ \mathcal{R} \cup \Sigma \big \rangle \cong \{1\}.$ \end{lemma}

\begin{proof} Let $\{\overline{x}_1,\ldots,\overline{x}_r\}\subset Q$, $\overline{\Sigma}\subset Q$ and $\overline{\sigma}\in Q$ be the image of $\{x_1,\ldots,x_r\}$, $\Sigma\subset \mathsf{F}(\{x_1,\ldots,x_r\})$ and $\sigma \in \Sigma$ in $Q$ respectively. We first observe that the finitely generated subgroup \hbox{$W_{\Sigma,k}'$ of $Q^{k+1} $,} $$W_{\Sigma,k}':=\left\{\begin{matrix}[1.7]
\Big \langle \big\{(\overline{x}_1,\overline{x}_1),\ldots, (\overline{x}_r, \overline{x}_r), (1,\overline{\sigma}):  \overline{\sigma}\in \overline{\Sigma}\big\}  \Big \rangle, \  k=1\ \ \ \ \ \ \ \ \ \ \ \ \ \ \ \ \ \ \ \ \ \ \ \ \ \ \ \ \ \ \ \ \ \ \ \ \ \ \ \  \ \  & \\ 
\Big\langle\big\{(\overline{x}_1,\ldots,\overline{x}_1),\ldots,(\overline{x}_r,\ldots, \overline{x}_r), (1,\overline{\sigma},\ldots,1), \ldots, (1,1,\ldots, \overline{\sigma}): \overline{\sigma}\in \overline{\Sigma}\big\} \Big \rangle, \ k>1& 
\end{matrix}\right.$$ is a quotient of $W_{\Sigma,k}$. To see this, consider the homomorphism $\pi_0:W_{\Sigma,k}\rightarrow W_{\Sigma,k}'$ induced by the following map on the set of generators of $W_{\Sigma,k}$: $$x_i \mapsto  (\overline{x}_i,\ldots, \overline{x}_i),\ y_{j\sigma}\mapsto \big(\overbrace{1,\ldots,1}^{j}, \overline{\sigma}, \overbrace{1,\ldots,1}^{k-j}\big)$$ for $1\leq i \leq r,\ 1\leq j \leq k, \ \sigma \in \Sigma.$ Observe that $\pi_0:W_{\Sigma,k}\rightarrow W_{\Sigma,k}'$ is well-defined, onto and for $1\leq j \leq k$, $\{1\}^{j}\times \llangle \overline{\Sigma}\rrangle\times \{1\}^{k-j}$ is a subgroup of $W_{\Sigma,k}'$. If $\mathsf{F}(\{x_1,\ldots,x_r\})=\llangle \mathcal{R}\cup \Sigma\rrangle$, or equivalently $Q=\llangle \overline{\Sigma} \rrangle$, then  $W_{\Sigma,k}'=Q^{k+1}$ and $\mathsf{Epi}(W_{\Sigma,k},Q^{k+1})\neq \emptyset$.

Now we use induction on $k\in \mathbb{N}$ to prove the implication  \begin{equation}\label{implication}\mathsf{Epi}\big(W_{\Sigma,k},Q^{k+1} \big) \neq \emptyset \Longrightarrow Q=\llangle \overline{\Sigma}\rrangle.\end{equation} First we prove (\ref{implication}) for $k=1$. Suppose there is an epimorphism $\phi_1:W_{\Sigma,1}\rightarrow{} Q\times Q$ and consider the commuting subgroups of $W_{\Sigma,1}$: $$\mathsf{H}_{11}:=\big \langle \{y_{1\sigma}:\sigma \in \Sigma\}\big \rangle, \ \mathsf{H}_{11}':=\big \langle \{\sigma^{-1}y_{1\sigma}:\sigma \in \Sigma\}\big \rangle$$ such that $W_{\Sigma,1}/\llangle \mathsf{H}_{11}\rrangle \cong  W_{\Sigma,1}/\llangle \mathsf{H}_{11}'\rrangle \cong  Q$.

For the rest of the proof of the lemma, fix a surjective homomorphism $\theta:Q\rightarrow F_2$ onto the free group of rank $2$. Let also $\pi_{\textup{L}}: Q\times Q\rightarrow Q$, $\pi_{\textup{R}}: Q\times Q\rightarrow Q$ be the projections onto the left and right direct factors respectively We will need the following claim.

\begin{claim}\label{claim0} $W_{\Sigma,1}=\llangle \mathsf{H}_{11},\mathsf{H}_{11}'\rrangle$.\end{claim}

\begin{proof}[Proof of Claim \ref{claim0}] First, let us observe that \begin{align*} \phi_1(\mathsf{H}_{11})&\subset \pi_{\textup{L}}(\phi_1(\mathsf{H}_{11}))\times \pi_{\textup{R}}(\phi_1(\mathsf{H}_{11}))\\  
\phi_1(\mathsf{H}_{11}')&\subset \pi_{\textup{L}}(\phi_1(\mathsf{H}_{11}'))\times \pi_{\textup{R}}(\phi_1(\mathsf{H}_{11}'))\end{align*} and note that $\pi_{\textup{L}}(\phi_1(\mathsf{H}_{11}))$ (resp. $\pi_{\textup{R}}(\phi_1(\mathsf{H}_{11}))$) centralizes $\pi_{\textup{L}}(\phi_1(\mathsf{H}_{11}'))$ (resp. $\pi_{\textup{R}}(\phi_1(\mathsf{H}_{11}'))$). Consider the subgroups of $Q$, \begin{align*}\mathsf{L}_1&:=\big\langle \pi_{\textup{L}}(\phi_1(\mathsf{H}_{11})), \pi_{\textup{L}}(\phi_1(\mathsf{H}_{11}'))\big\rangle\\ \mathsf{L}_2&:=\big\langle \pi_{\textup{R}}(\phi_1(\mathsf{H}_{11})), \pi_{\textup{R}}(\phi_1(\mathsf{H}_{11}'))\big\rangle.\end{align*}

We claim that both $\theta(\mathsf{L}_1)$ and $\theta(\mathsf{L}_2)$ are non-cyclic free subgroups of $F_2$. To see this, first observe that since $\phi_1$ and $\theta$ are surjective, we obtain the epimorphism $$W_{\Sigma,1}/\llangle \mathsf{H}_{11},\mathsf{H}_{11}'\rrangle \rightarrow{} Q/\llangle \mathsf{L}_1 \rrangle \times Q/\llangle \mathsf{L}_2 \rrangle \rightarrow{} F_2/\llangle \theta(\mathsf{L}_1) \rrangle \times F_2/\llangle \theta(\mathsf{L}_2) \rrangle.$$ If $\theta(L_i)$ is cyclic for some $i=1,2$, then $b_1(F_2/\llangle \theta(\mathsf{L}_i) \rrangle)\geq 1$. However, this is absurd since $$W_{\Sigma,1}/\llangle \mathsf{H}_{11},\mathsf{H}_{11}'\rrangle=\big\langle x_1,\ldots,x_r \ | \ \mathcal{R}\cup \Sigma \big\rangle$$ has zero first Betti number. Hence, both $\theta(\mathsf{L}_1),\theta(\mathsf{L}_2)$ are non-cyclic free groups. In particular, since $\theta(\pi_{\textup{L}}(\phi_1(\mathsf{H}_{11})))$ centralizes $\theta(\pi_{\textup{L}}(\phi_1(\mathsf{H}_{11}')))$, precisely one of these two groups is trivial and the other is free. By our assumption that centralizers of non-cyclic free subgroups of $Q$ are trivial, we conclude that precisely one of the commuting subgroups $\pi_{\textup{L}}(\phi_1(\mathsf{H}_{11}))$ and $\pi_{\textup{L}}(\phi_1(\mathsf{H}_{11}'))$ of $Q$ contains a non-cyclic free subgroup and the other is trivial. Similarly, we deduce that one of the groups $\pi_{\textup{R}}(\phi_1(\mathsf{H}_{11}))$ and $\pi_{\textup{R}}(\phi_1(\mathsf{H}_{11}'))$ contains a non-cyclic free subgroup and the other is trivial. Observe that $W_{\Sigma,1}/\llangle \mathsf{H}_{11}\rrangle \cong Q$, $W_{\Sigma,1}/\llangle \mathsf{H}_{11}'\rrangle \cong Q$ and there is no epimorphism of $Q$ onto $Q\times Q$ since $Q$ is a Hopfian group.
Thus, by the previous remarks, up to composing with the automorphism of $Q\times Q$ switching the two factors, there are subgroups $\mathsf{N}_1,\mathsf{N}_2$ of $Q$ such that $$\phi_1(\mathsf{H}_{11})=\{1\}\times \mathsf{N_1}, \ \phi_1(\mathsf{H}_{11}')=\mathsf{N}_2\times \{1\}.$$ 

Now consider the epimorphism $\overline{\phi}_1:W_{\Sigma,1}/\llangle \mathsf{H}_{11}\rrangle \rightarrow Q\times \big(Q/\llangle \mathsf{N}_1\rrangle\big)$, $$\overline{\phi}_1\big(g\llangle \mathsf{H}_{11}\rrangle\big)=\phi_1(g)\llangle \phi(\mathsf{H}_{11})\rrangle.$$ The homomorphism $\pi_{\textup{L}} \circ \overline{\phi}:W_{\Sigma,1}/\llangle \mathsf{H}_{11}\rrangle \rightarrow Q$, where $\pi_{\textup{L}}:Q\times \big( Q/\llangle \mathsf{N}_1\rrangle\big) \rightarrow Q$ is the projection onto the first factor, is onto and hence an isomorphism since $W_{\Sigma,1}/\llangle \mathsf{H}_{11}\rrangle \cong Q$ and $Q$ is Hopfian. In particular, $Q=\llangle \mathsf{N}_1\rrangle$, $\overline{\phi}$ is an isomorphism and $\textup{ker}\phi$ is a subgroup of $\llangle \mathsf{H}_{11}\rrangle$. Similarly, there is an epimorphism $$Q\cong W_{\Sigma,1}/\llangle \mathsf{H}_{11}'\rrangle \rightarrow{} \big(Q/\llangle \mathsf{N}_2\rrangle\big) \times Q\rightarrow{}  Q$$ and hence $Q=\llangle \mathsf{N}_2\rrangle$ and $\textup{ker}\phi$ is a subgroup of $\llangle \mathsf{H}_{11}\rrangle$. Since $\phi$ is onto and $\textup{ker}\phi$ is contained in $\llangle \mathsf{H}_{11},\mathsf{H}_{11}'\rrangle$, we have $W_{\Sigma,1}=\llangle \mathsf{H}_{11},\mathsf{H}_{11}'\rrangle$.\end{proof}

Now we complete the proof of (\ref{implication}) for $k=1$. By Claim \ref{claim0}, $W_{\Sigma,1}=\llangle  \mathsf{H}_{11},\mathsf{H}_{11}'\rrangle$, hence by using the given presentation of $W_{\Sigma,1}$ we have \begin{align*}\mathsf{F}\big(\{y_{1\sigma}:\sigma \in \Sigma\}\cup \{x_1,\ldots,x_r\}\big)&=\llangle[\big] \mathcal{R}\cup \{\sigma^{-1}y_{1\sigma}:\sigma \in \Sigma\}\cup  \{y_{1\sigma}:\sigma \in \Sigma\} \rrangle[\big]\\ &=\llangle[\big] \mathcal{R}\cup \Sigma \cup  \{y_{1\sigma}:\sigma \in \Sigma\} \rrangle[\big]\end{align*} which implies $\big \langle x_1,\ldots,x_r \ | \ \mathcal{R} \cup \Sigma \big \rangle \cong \{1\}$ and (\ref{implication}) follows for $k=1$.
\medskip

\noindent {\em The case where $k\geq 2$.} Now assume there is an epimorphism $\phi_k:W_{\Sigma,k}\rightarrow Q^{k+1}$. For $1\leq j \leq k$ consider the subgroup $\mathsf{H}_{kj}\subset W_{\Sigma, k}$: $$\mathsf{H}_{kj}:=\big \langle \{y_{j\sigma}:\sigma \in \Sigma\}\big\rangle, \ \mathsf{H}_{kj}':= \big\langle \{\sigma^{-1}y_{j\sigma}:\sigma \in \Sigma\}\big\rangle.$$

\begin{claim}\label{claim-k} $\phi_{k}(\mathsf{H}_{k1})$ is a subgroup of $\{1\}^{m}\times Q\times \{1\}^{k-m}$ for some $m=0,\ldots,k$. \end{claim}

\begin{proof} For $j=1,\ldots,k$ and $i=1,\ldots,k+1$, let $\mathsf{M}_{ij}$ (resp. $\mathsf{M}_{ij}'$) be the projection of $\phi_{k}(\mathsf{H}_{kj})$ onto the $i^{\textup{th}}$ direct factor of $Q^{k+1}$. Since $\mathsf{H}_{ks}'$ commutes with $\mathsf{H}_{ks}$, $\mathsf{H}_{kj}$ commutes with $\mathsf{H}_{ki}$ for $i\neq j$ and centralizers of non-cyclic free subgroups in $Q$ are trivial,  for every $s=1,\ldots,k$ at most one of the groups $\mathsf{M}_{1s},\ldots,\mathsf{M}_{(k+1)s}$ contains a non-cyclic free group. Indeed, if both $\mathsf{M}_{j_1 s}$ and $\mathsf{M}_{j_2s}$ contain non-cyclic free groups for some pair $j_1\neq j_2$, then $\mathsf{M}_{j_1r}=\mathsf{M}_{j_2r}=\{1\}$ for $r\neq s$ and $\mathsf{M}_{j_1s}'=\mathsf{M}_{j_2s}'=\{1\}$. This implies that $\phi_k(\langle \mathsf{H}_{ks}'\cup \bigcup_{r\neq s} \mathsf{H}_{kr}\rangle)$ projects trivially on the $j_1^{\textup{th}}$ and $j_2^{\textup{th}}$ direct factors of $Q^{k+1}$ and the quotient $Q\cong W_{\Sigma,k}/\llangle \mathsf{H}_{ks}'\cup \bigcup_{r\neq s} \mathsf{H}_{kr}\rrangle$ maps onto $Q\times Q$. However, this is impossible since $Q$ is a Hopfian group.

If the claim is not true, up to composing with an automorphism of $Q^{k+1}$ permuting the direct factors, we may assume that $\phi_k(\mathsf{H}_{k1})$ projects non-trivially on the first two direct factors of $Q^{k+1}$. Let us set $\mathsf{M}_{1}':=\langle \mathsf{M}_{11}',\mathsf{M}_{11},\ldots,\mathsf{M}_{1(k+1)}\rangle$ and $\mathsf{M}_{2}':=\langle \mathsf{M}_{21}', \mathsf{M}_{21}, \ldots,\mathsf{M}_{2(k+1)}\rangle$. Since $\phi_k$ and $\theta$ are surjective and $\phi(\langle \mathsf{H}_{k1}',\mathsf{H}_{k1},\ldots,\mathsf{H}_{kk}\rangle)$ is a subgroup of $\mathsf{M}_1'\times \mathsf{M}_2'\times Q^{k-1}$, there are epimorphisms $$W_{\Sigma,k}/\llangle \mathsf{H}_{k1}',\mathsf{H}_{k1},\ldots,\mathsf{H}_{kk}\rrangle \rightarrow{} Q/\llangle \mathsf{M}_1'\rrangle \times Q/\llangle \mathsf{M}_2'\rrangle  \rightarrow{} F_2/\llangle \theta(\mathsf{M}_1')\rrangle \times F_2/\llangle \theta(\mathsf{M}_2')\rrangle.$$ By assumption, $$W_{\Sigma,k}/\llangle \mathsf{H}_{k1}',\mathsf{H}_{k1},\ldots,\mathsf{H}_{kk}\rrangle\cong \big\langle x_1,\ldots,x_r \ | \ \mathcal{R}\cup \Sigma \big\rangle$$ has zero first Betti number, thus $b_1(F_2/\llangle \theta(\mathsf{M}_i')\rrangle )=0$ for $i=1,2$ and both $\theta(\mathsf{M}_1'),\theta(\mathsf{M}_2')$ contain a non-cyclic free group. Note that since $\mathsf{M}_{11}\neq \{1\}$ centralizes $\langle \mathsf{M}_{11}',\ldots,\mathsf{M}_{1(k+1)}\rangle\subset Q$, the latter group (and its image under $\theta$) cannot contain a non-cyclic free group. Since $\theta(\mathsf{M}_{11})$ centralizes $\langle \theta(\mathsf{M}_{11}'),\theta(\mathsf{M}_{12}),\ldots,\theta(\mathsf{M}_{1(k+1)})\rangle$ and $\theta(\mathsf{M}_{1}')$ is not cyclic, we deduce that $\theta(\mathsf{M}_{11})$ and $\mathsf{M}_{11}$ contain a non-cyclic free subgroup and hence $\mathsf{M}_{11}'=\mathsf{M}_{12}=\cdots=\mathsf{M}_{1(k+1)}=\{1\}$. Similarly, as $\mathsf{M}_{21}\neq \{1\}$ and $\theta(\mathsf{M}_2')$ is not cyclic, we deduce that $\mathsf{M}_{21}$ contains a non-cyclic free group. However, this contradicts our observation in the previous paragraph that at most one of the groups $\mathsf{M}_{11},\mathsf{M}_{21},\ldots, \mathsf{M}_{(k+1)1}$ can contain a non-cyclic free subgroup. The claim follows.\end{proof}

By Claim \ref{claim-k}, up to composing $\phi_k$ with an automorphism of $Q^{k+1}$ permuting the factors, we may assume $\phi_k(\mathsf{H}_{k1})=\mathsf{N}_{k1}\times \{1\}^{k}$. In particular, there is an epimorphism $$W_{\Sigma,k}/\llangle \mathsf{H}_{k1}\rrangle \rightarrow \big(Q/\llangle \mathsf{N}_{k1}\rrangle\big) \times Q^{k}\rightarrow Q^{k}.$$ Finally, observe that since $W_{\Sigma,k}/\llangle \mathsf{H}_{k1}\rrangle \cong W_{\Sigma, k-1}$, there is an epimorphimsm $W_{\Sigma,k-1}\rightarrow  Q^{k}$ and thus by the inductive step we have $Q=\llangle \overline{\Sigma}\rrangle$. This completes the proof of the induction and (\ref{implication}) follows.\end{proof}

\subsection{Undecidability}  For some background on decision problems we refer the reader to \cite{Miller}. We will need the following undecidability result for the construction of the recursive sequence of hyperbolic groups in Theorem \ref{mainthm2}.

\begin{theorem}\label{undecidability}\textup{(Miller \cite{Miller})} There exists a recursive sequence of finite subsets $(\Sigma_n)_{n\in \mathbb{N}}$ of the free group $\mathsf{F}(\{x,y\})$, of fixed cardinality, with the following properties:
\begin{enumerate}
\item for every $n\in \mathbb{N}$ the group $\langle x,y \ | \ \Sigma_n\rangle$ has zero first Betti number.
\item for every $n\in \mathbb{N}$ the group $\langle x,y \ |\ \Sigma_n\rangle$ is either trivial or infinite.
\item the set $\{n\in \mathbb{N}: \langle x,y \ |\ \Sigma_n\rangle\cong \{1\} \}$ is recursively enumerable but not recursive.\end{enumerate}\end{theorem}

\begin{rmk} \normalfont{The finite sets $(\Sigma_n)_{n\in \mathbb{N}}$ in Theorem \ref{undecidability} are obtained by applying \cite[Lem. 3.6]{Miller} to a finitely presented group $\big \langle y_1,\ldots,y_s \ | \ \omega_1,\ldots, \omega_m \big \rangle$ with unsolvable word problem. Let $a:=y^{-1}x[y,x], c:=x[y,x]$ and consider the words of $\mathsf{F}(\{x,y\})$, \begin{align*} b_0&:=a^{-2}x^{-2}axa^2c^{-2}x^{-1}c^{-1}xc^2\\ b_i&:=[a^{3+i}c^{-3-i},x], \ i=1,\ldots,s. \end{align*}
 For every $i$, write $\omega_i=\omega_i(y_1,\ldots,y_s)$ in reduced form and for a word $w=w(y_1,\ldots,y_s)$ in $\mathsf{F}(\{y_1,\ldots,y_s\})$ define the set $$\Sigma_{w}:=\Big\{b_0,\omega_1(b_1,\ldots,b_s),\ldots, \omega_m(b_1,\ldots,b_s), [w(b_1,\ldots,b_s),x]a^3c^{-3}x^{-1}c^3a^{-3}\Big\}.$$ By \cite[Lem. 3.6]{Miller} the presentation $\langle x,y \ | \ \Sigma_w\rangle$ has trivial abelianization and defines the trivial group if and only if $w\in \llangle \omega_1,\ldots, \omega_m \rrangle$.} \end{rmk}

Now we can prove Theorems \ref{mainthm2} and \ref{maincor}.

\begin{proof}[Proof of Theorem \ref{mainthm2} \& Theorem \ref{maincor}.] 

Let $Q=\langle X\ |\ \mathcal{R}\rangle$ be a finite presentation of $Q$. By assumption, we can choose an epimorphism $\pi_Q:\mathsf{F}(X) \rightarrow \mathsf{F}(\{a,b\})$ with $\textup{ker}\pi_Q=\llangle \mathcal{R}\cup \mathcal{R}' \rrangle$ for some finite subset $\mathcal{R}'\subset \mathsf{F}(X)$. Fix $\overline{a},\overline{b}\in \mathsf{F}(X)$ with $\pi_Q(\overline{a})=a$ and $\pi_Q(\overline{b})=b$. Let $(\Sigma_n)_{n\in \mathbb{N}}$, $\Sigma_n \subset \mathsf{F}(\{a,b\})$, be a recursive sequence of finite subsets furnished by Theorem \ref{undecidability} and consider the recursive sequence $(\Sigma_n')_{n\in \mathbb{N}}$ of finite subsets of $\mathsf{F}(X)$ defined as follows: $$\Sigma_n':=\mathcal{R}'\cup \big\{w(\overline{a},\overline{b}):w(a,b)\in \Sigma_n\big\}.$$ Note that for every $n\in \mathbb{N}$, $\pi_Q(\Sigma_n')=\Sigma_n$ and $\langle X\ |\ \mathcal{R}\cup \Sigma_n'\rangle \cong \langle a,b \ | \ \Sigma_n\rangle$. In particular, by Theorem \ref{undecidability} (iii), the set $\big\{n\in \mathbb{N}: \langle X\ |\ \mathcal{R}\cup \Sigma_n'\rangle \cong \{1\}\big\}$ is recursively enumerable but not recursive.

Fix an integer $k\geq 1$ and consider the group $W_{\Sigma_n'}$ defined by the presentation: $$W_{\Sigma_n'}:=\Big \langle x_1,\ldots,x_r, \{y_{1\sigma}\}_{\sigma \in \Sigma_n'}, \ldots, \{y_{k\sigma}\}_{\sigma \in \Sigma_n'}  \ \big| \ \mathcal{R}, [\sigma^{-1}y_{s\sigma}, y_{s\sigma'}], [y_{s\sigma},y_{t\sigma'}], \sigma,\sigma'\in \Sigma_n', s\neq t \Big \rangle.$$ 

Clearly, since $(\Sigma_n')_{n\in \mathbb{N}}$ is recursive, the sequence of presentations $(W_{\Sigma_n'})_{n\in \mathbb{N}}$ is also recursive. Since $Q$ is torsion-free Hopfian \cite{Sela} and $\langle X\ | \ \mathcal{R}\cup \Sigma_n'\rangle$ has zero first Betti number, by Lemma \ref{mainlemma}, $\mathsf{Epi}(W_{\Sigma_n'},Q^{k+1})\neq \emptyset$  if and only if $\langle X\ |\ \mathcal{R}\cup \Sigma_n'\rangle\cong \{1\}$. In particular, the sets $$\big\{n\in \mathbb{N}: \mathsf{Epi}(W_{\Sigma_n'},Q^{k+1})\neq \emptyset \big\}=\big\{n\in \mathbb{N}: \langle X\ |\ \mathcal{R}\cup \Sigma_n'\rangle \cong \{1\}\big\}=\big\{n\in \mathbb{N}: \langle a,b \ | \ \Sigma_n\rangle \cong \{1\}\big\}$$ are all recursively enumerable but not recursive. This completes the proof of Theorem \ref{maincor}.

Now we also assume that $Q$ is virtually residually solvable. By applying the algorithm of Proposition \ref{Rips0} with input the presentation of $W_{\Sigma_n'}$, we obtain the finitely presented group \begin{align}\label{Gamma-pres}\Gamma_{n}&=\big \langle  x,y,x_1,\ldots,x_r, \{y_{1\sigma}\}_{\sigma \in \Sigma_n'}, \ldots, \{y_{k\sigma}\}_{\sigma \in \Sigma_n'}\ \big | \ \mathcal{R}_n \big \rangle\\ \nonumber |\mathcal{R}_n|&=5+|\mathcal{R}|+4r+4k|\Sigma_n'|+ \frac{k^2+k}{2}|\Sigma_n'|^2,\end{align} satisfying the $C'(\frac{1}{6})$ small cancellation condition, and an epimorphism $\pi_{n}:\Gamma_{n}\rightarrow W_{\Sigma_n'}$, with kernel $\langle x,y\rangle$, such that any group homomorphism \hbox{$\rho:\Gamma_{n}\rightarrow Q^{k+1}$} factors through $\pi_{n}$. Clearly, since the sequence of presentations $(W_{\Sigma_n'})_{n\in \mathbb{N}}$ is recursive, the sequence $(\Gamma_{n})_{n\in \mathbb{N}}$ in (\ref{Gamma-pres}) is also recursive. In particular, $\mathsf{Epi}(\Gamma_{n},Q^{k+1})\neq \emptyset$ if and only if $\mathsf{Epi}(W_{\Sigma_{n}'},Q^{k+1})\neq \emptyset$. Since $Q$ is Hopfian and $\langle X\ | \ \mathcal{R}\cup \Sigma_n'\rangle$ has zero first Betti number, by Lemma \ref{mainlemma}, we have $$\big\{n\in \mathbb{N}: \mathsf{Epi}(\Gamma_{n},Q^{k+1})\neq \emptyset \big\}=\big\{n\in \mathbb{N}: \langle X\ |\ \mathcal{R}\cup \Sigma_n'\rangle \cong \{1\}\big\}=\big\{n\in \mathbb{N}: \langle a,b \ | \ \Sigma_n\rangle \cong \{1\}\big\}.$$ All previous sets are recursively enumerable but not recursive. Therefore, the sequence $(\Gamma_{n})_{n\in \mathbb{N}}$ satisfies the conclusion of the Theorem \ref{mainthm2}.\end{proof}

\subsection{Undecidability of detecting epimorphisms onto fiber products}
In this subsection, we prove Theorem \ref{fiberprod}. Recall that a group $H$ is called freely indecomposable if it does not split as a free product of two non-trivial groups. We will need the following observation.

\begin{lemma}\label{freeprod} Let $H$ be a freely indecomposable group and $h\in H$. If the amalgam $H\ast_{h=h}H$ is isomophic to a non-trivial free product $G_1\ast G_2$, then $h=1$.\end{lemma}

\begin{proof} Suppose there is an isomorphism $\psi: H_1\ast_{h_1=h_2}H_2\rightarrow G_1\ast G_2$, where $H_i$ (resp. $h_i$) is a copy of $H$ (resp. $h$). Since $H_i$ is freely indecomposable, by the Kurosh subgroup theorem, up to conjugating $\psi$ by an element of $G_1\ast G_2$, we may assume that $\psi(H_1)\subset G_i$ and $\psi(H_2)\subset gG_j g^{-1}$ for some $g\in G_1\ast G_2$. Note that necessarily $i\neq j$, since $\llangle \psi(H_1),\psi(H_2)\rrangle=G_1\ast G_2$ and $(G_1\ast G_2)/\llangle G_i \rrangle\cong G_{3-i}$ for $i\in \{1,2\}$. It follows that $i\neq j$ and $G_i \cap gG_{j}g^{-1}$ contains $\psi(H_1\cap H_2)=\langle \psi(h_1)\rangle=\langle \psi(h_2)\rangle$. Since $G_i\cap gG_jg^{-1}=\{1\}$ and $\psi$ is injective it follows that $h=1$.\end{proof}

\begin{proof}[Proof of Theorem \ref{fiberprod}] Let $G=\big \langle x_1,\ldots,x_r \ | \ \mathcal{S}\big\rangle$ be a group of type $\mathcal{F}_3$ with unsolvable word problem and no non-trivial finite quotients. We consider the following presentation of the free product $G\ast \overline{G}=\big \langle x_1,\ldots,x_{2r} \ | \ \mathcal{S}, \overline{\mathcal{S}}\big\rangle$, where $\overline{G}$ is a copy of $G$, $x_{i+r}$ is a copy of $x_i$ for every $1\leq i \leq r$, and $\overline{\mathcal{S}}$ is a copy of $\mathcal{S}$ seen as a subset of $\mathsf{F}(\{x_{r+1},\ldots,x_{2r}\})$.

Fix a short exact sequence $1\rightarrow N \rightarrow \Gamma \overset{\pi}{\rightarrow} W\rightarrow 1$, where $\Gamma$ is torsion-free, linear and hyperbolic. We may choose a presentation of $\Gamma$ of the form (see the discussion in \cite{BBMS}),

 \begin{equation*} \Gamma=\scalebox{2.2}{\Bigg \langle}\!\!\!\begin{matrix}[1.25]
 x_1,\ldots, x_{2r}\\
y_1,\ldots,y_{k}\\ t
 \end{matrix}  \ \scalebox{2.1}{\Bigg{|}} \ \begin{matrix}[1.25] A_{s}(y_1,\ldots,y_k), s \in J_1\\ x_{i}^{\varepsilon}y_\ell x_{i}^{-\varepsilon}B_{i\ell \varepsilon}(y_1,\ldots,y_k)\ \ \\ 
V_{j}\big(x_1,\ldots,x_{2r}, y_1,\ldots,y_k \big),\ j \in J_2\\
i=1,\ldots,2r,\ \ell=1\ldots,k, \ \varepsilon=\pm 1\\
\vspace{0.001cm} \end{matrix} \scalebox{2.1}{\Bigg \rangle}\end{equation*} where $\{V_{j}\}_{j\in J}$ are words in $\mathsf{F}(\{x_1,\ldots,x_{2r},y_1,\ldots,y_k\})$, $A_{s},B_{i\ell \varepsilon}\in \mathsf{F}(\{y_1\ldots,y_k\})$ such that $\mathcal{S}\cup \overline{\mathcal{S}}=\big\{\mathsf{V}_j(x_1,\ldots,x_{2r},1,\ldots,1): j\in J\big\}$.

For every word $w=w(x_1,\ldots,x_r)\in \mathsf{F}(\{x_1,\ldots,x_r\})$ consider the word in $\mathsf{F}(\{x_1,\ldots,x_{2r}\})$, \begin{align}\label{widetilde-w}\widetilde{w}:=w(x_1,\ldots,x_r)w(x_{r+1},\ldots,x_{2r})^{-1}\end{align} and the finitely presented group

\begin{equation}\label{presentation-G} \mathcal{G}_w:=\scalebox{2.9}{\Bigg \langle} \! \!\!  \begin{matrix}[1.25]
 x_1,\ldots, x_{2r} \ \ \\
y_1^{\textup{L}},\ldots, y_k^{\textup{L}}\ \ \\
y_1^{\textup{R}},\ldots,y_k^{\textup{R}}\ \ \\ t \end{matrix}   \scalebox{2.8}{\Bigg{|}}  \begin{matrix}[1.25] \big[\widetilde{w}t,t\big]\\
\big[t,y_\ell^{\textup{L}}\big]\\ \big[\widetilde{w}t,y_\ell^{\textup{R}}\big]\\
\big[y_\ell^{\textup{R}},y_m^{\textup{L}}\big]\\
\ell,m=1,\ldots,k\\
\vspace{0.001cm}\end{matrix}\ \  \ \begin{matrix}[1.25] 
x_{i}^{\varepsilon}y_\ell^{\textup{L}}x_{i}^{-\varepsilon}B_{i\ell \varepsilon}(y_1^{\textup{L}},\ldots,y_k^{\textup{L}})\\
  x_{i}^{\varepsilon}y_\ell^{\textup{R}}x_{i}^{-\varepsilon}B_{i\ell\varepsilon}(y_1^{\textup{R}},\ldots,y_k^{\textup{R}})\\

A_{s}(y_1^{\textup{L}},\ldots,y_k^{\textup{L}}), A_{s}(y_1^{\textup{R}},\ldots,y_k^{\textup{R}})\\ 
V_{j}\big(x_1,\ldots,x_r, y_1^{\textup{L}}y_1^{\textup{R}},\ldots,y_k^{\textup{L}}y_k^{\textup{R}} \big)\\ 
 s\in J_1,\ j\in J_2,\ i=1,\ldots,2r,\ \varepsilon=\pm 1\\
\vspace{0.001cm} \end{matrix}\!\! \scalebox{2.9}{\Bigg \rangle}. \end{equation}

Let $P:=\big\{(\gamma_1,\gamma_2)\in \Gamma \times \Gamma: \pi(\gamma_1)=\pi(\gamma_2)\big\}$ be the fiber product of $\Gamma$ with respect to $N$ and $\pi_{\textup{L}}:P\rightarrow \Gamma$, $\pi_{\textup{R}}:P\rightarrow \Gamma$ the projections onto $\Gamma$. Denote by $\llangle \mathcal{S}\rrangle$ the normal closure of $\mathcal{S}$ in $\mathsf{F}(\{x_1,\ldots,x_r\})$. We have the following claim.

\begin{claim}\label{epi-prod} Let $w\in \mathsf{F}(\{x_1,\ldots,x_r\})$. Then $\mathsf{Epi}(\mathcal{G}_w,P)\neq \emptyset$ if and only if $w\in \llangle \mathcal{S}\rrangle$.\end{claim}

\begin{proof} Let $\overline{x}_i, \overline{y}_j, w', \widetilde{w}'\in \Gamma$ be the images of $x_i,y_j, w$ and $\widetilde{w}$ in $\Gamma$ respectively. If $w\in \llangle \mathcal{S}\rrangle$, then clearly $w',\widetilde{w}'\in N$. The map $x_{i}\mapsto (\overline{x}_i,\overline{x}_i)$, $y_j^{\textup{L}}\mapsto (\overline{y}_j,1)$, $y_j^{\textup{R}}\mapsto (1,\overline{y}_j)$ and $t\mapsto(1, (\widetilde{w}')^{-1})\in \{1\}\times N$ induces a well-defined epimorphism from $\mathcal{G}_w$ onto $P$.

Now let us assume there exists an epimorphism $\varphi:\mathcal{G}_w\rightarrow P$. By the definition of the presentation of $\mathcal{G}_w$, $N^{\textup{L}}:=\big\langle y_1^{\textup{L}},\ldots,y_k^{\textup{L}}\big\rangle$ and $N^{\textup{R}}:=\big\langle y_1^{\textup{R}},\ldots,y_k^{\textup{R}}\big\rangle$ are normal commuting subgroups of $\mathcal{G}_w$. First, note that
$\mathcal{G}_{w}/\langle N^{\textup{L}},N^{\textup{R}}\rangle\cong \big \langle x_1,\ldots,x_{2r},t \ |\ [\widetilde{w}t,t], \mathcal{S},\overline{\mathcal{S}}\big\rangle$ is a quotient of $G\ast G\ast \mathbb{Z}$. Since $G$ has no non-trivial finite quotients and $P$ is residually finite and non-cyclic, it follows that at least one of the groups $\varphi(N^{\textup{L}})$ or $\varphi(N^{\textup{R}})$ is non-trivial. 

Without loss of generality, we may assume that $\varphi(N^{\textup{L}})\neq \{1\}$. Up to composing $\varphi$ with the automorphism of $P$ switching the factors we may assume that $\pi_{\textup{L}}(\varphi(N^{\textup{L}}))$ is a non-elementary normal subgroup of $\Gamma$. Consider the commuting subgroups of $\mathcal{G}_w$, $$\mathsf{H}_{\textup{L}}=\big\langle N^{\textup{L}}, \widetilde{w}t\big\rangle, \ \mathsf{H}_{\textup{R}}=\big\langle N^{\textup{R}}, t\big\rangle$$ and observe that $\mathcal{G}_w/\llangle \mathsf{H}_{\textup{L}}\rrangle\cong \Gamma$ and $\mathcal{G}_w/\llangle \mathsf{H}_{\textup{R}}\rrangle\cong \Gamma$. Since $\pi_{\textup{L}}(\varphi(\mathsf{H}_{\textup{R}}))$ centralizes $\pi_{\textup{L}}(\varphi(N^{\textup{L}}))$, it follows that $\varphi(\mathsf{H}_{\textup{R}})=\{1\}\times N_2$ for some subgroup $N_2$ of $N$. In addition, since $N$ is normal in $\Gamma$, $\llangle \varphi(\mathsf{H}_{\textup{R}})\rrangle\subset \{1\}\times N$. Therefore, as $\varphi$ is onto, we obtain the well-defined epimorphism $$\varphi^{\textup{R}}:\mathcal{G}_w/\llangle \mathsf{H}_{\textup{R}}\rrangle \rightarrow{} P/\llangle \varphi(\mathsf{H}_{\textup{R}})\rrangle \rightarrow{} P/\big(\{1\}\times N\big),$$ $\varphi^{\textup{R}}(g\llangle \mathsf{H}_{\textup{R}}\rrangle)=\varphi(g)(1\times N), \ g \in \mathcal{G}_w$. Since $\Gamma$ is torsion-free, Hopfian and $P/(\{1\}\times N) \cong \Gamma$, we conclude that $\llangle \varphi(\mathsf{H}_{\textup{R}})\rrangle=\{1\}\times N$, $\varphi^{\textup{R}}$ is an isomorphism and $\textup{ker}\varphi\subset \llangle \mathsf{H}_{\textup{R}}\rrangle$.

In particular, $N_2$ is a non-trivial subgroup of $N$. First, let us verify that $N_2$ is non-cyclic. Indeed, if $N_2$ were cyclic, as $\pi_{\textup{R}}(\varphi(N^{\textup{R}}))$ is a normal subgroup of $\Gamma$ contained in $N_2$ (since $\varphi(N^{\textup{R}})\subset \varphi(\mathsf{H}_{\textup{R}})=\{1\}\times N_2$), we would have $\varphi(N^{\textup{R}})\subset N \times \{1\}$. In addition, the projection $\pi_{\textup{R}}(\varphi(N^{\textup{L}}))$ is a normal subgroup of $\Gamma$ centralizing $N_2$ and hence has to be trivial. This shows that $\varphi(\langle N^{\textup{L}},N^{\textup{R}}\rangle)$ is contained in $N\times \{1\}$, thus, since $\varphi$ is onto, we obtain an epimorphism $\mathcal{G}_w/\langle N^{\textup{L}},N^{\textup{R}}\rangle\rightarrow P/(N\times \{1\})\cong \Gamma$. However, this contradicts the fact that $\mathcal{G}_w/\langle N^{\textup{L}},N^{\textup{R}}\rangle$ is a quotient of $G\ast G\ast \mathbb{Z}$ and $G$ has no non-trivial finite quotients, while $\Gamma$ is residually finite and non-cyclic.

It follows that $N_2$ is a non-elementary subgroup of $\Gamma$. Since $\pi_{\textup{R}}(\varphi(\mathsf{H}_{\textup{L}}))$ centralizes $N_2$, $\varphi(\mathsf{H}_{\textup{L}})=N_1\times \{1\}$ for some subgroup $N_1$ of $N$. By arguing as previously, we obtain an epimorphism $\Gamma\cong \mathcal{G}_w/\llangle \mathsf{H}_{\textup{L}}\rrangle \rightarrow{} P/\llangle\varphi(\mathsf{H}_{\textup{L}})\rrangle \rightarrow{} P/(N\times \{1\})\cong \Gamma$. As $\Gamma$ is a Hopfian group, it follows that $\llangle \varphi(\mathsf{H}_{\textup{L}})\rrangle=N\times \{1\}$ and also $\textup{ker}\varphi\subset \llangle \mathsf{H}_{\textup{L}}\rrangle$.

Therefore, the normal subgroup $\llangle \mathsf{H}_{\textup{L}},\mathsf{H}_{\textup{R}}\rrangle$ of $\mathcal{G}_w$ contains $\textup{ker}\varphi$ and $\varphi(\llangle \mathsf{H}_{\textup{L}},\mathsf{H}_{\textup{R}}\rrangle)=N\times N$. Putting all this together, the induced homomorphism $\varphi':\mathcal{G}_w/\llangle \mathsf{H}_{\textup{L}},\mathsf{H}_{\textup{R}}\rrangle\rightarrow P/(N\times N)$, $$\varphi'\big(g\llangle \mathsf{H}_{\textup{L}},\mathsf{H}_{\textup{R}}\rrangle\big)=\varphi(g)(N\times N), \ g\in \mathcal{G}_w$$ is an isomorphism. By the definition of the presentation of $\mathcal{G}_w$ and the word $\widetilde{w}$ in (\ref{presentation-G}) and (\ref{widetilde-w}) respectively, $\mathcal{G}_w/\llangle \mathsf{H}_{\textup{L}},\mathsf{H}_{\textup{R}}\rrangle= (G\ast \overline{G})/\llangle \widetilde{w}\rrangle\cong G\ast_{w=w}G$. Since $P/(N\times N)\cong G\ast G$, we conclude that $G\ast_{w=w}G\cong G\ast G$. By Lemma \ref{freeprod},  it follows that $w=_G1$ and the claim follows.\end{proof}

Now we conclude the proof of the theorem. For every word $w\in \mathsf{F}(\{x_1,\ldots,x_r\})$, given the presentation $\mathcal{G}_{w}$, since $P\subset \Gamma \times \Gamma$ is virtually residually nilpotent \cite{Platonov}, the algorithm of Proposition \ref{Rips0} ouputs a torsion-free hyperbolic group $\Delta_w$, satisfying the $C'(\frac{1}{6})$ small cancellation condition, $$\Delta_w=\big\langle Y \  | \ \mathcal{S}_w\big \rangle$$ with $|Y|=3+2k+2r$, $|\mathcal{S}_w|=10+10k+8r+8kr+k^2+2|J_1|+|J_2|$ and an epimorphism $\pi_w:\Delta_w\rightarrow \mathcal{G}_w$, with the property that any homomorphism $\Delta_w\rightarrow P$ factors through $\pi_w$ and a homomorphism $\mathcal{G}_w\rightarrow P$. By Claim \ref{epi-prod} and the choice of $\Delta_w$ we conclude that $$\big\{w\in \mathsf{F}(\{x_1,\ldots,x_r\}): \mathsf{Epi}(\Delta_w,P)\neq \emptyset\big\}=\big\{w\in \mathsf{F}(\{x_1,\ldots,x_r\}): \mathsf{Epi}(\mathcal{G}_w,P)\neq \emptyset\big\}=\llangle \mathcal{S}\rrangle.$$ Clearly, $\llangle \mathcal{S}\rrangle$ is a recursively enumerable subset of $\mathsf{F}(\{x_1,\ldots,x_r\})$ which is not recursive since $G$ has unsolvable word problem. This concludes the proof of the theorem.\end{proof}

\bibliographystyle{siam}

\bibliography{biblio.bib}

\end{document}